\documentclass[12pt]{amsart}
\usepackage{amscd,amsmath,amsthm,amssymb}
\usepackage[left]{lineno}
\usepackage{pstcol,pst-plot,pst-3d}%\usepackage[T1]{fontenc}
\usepackage{color}
\usepackage{pstricks}
\usepackage{stmaryrd}
\newpsstyle{fatline}{linewidth=1.5pt}
\newpsstyle{fyp}{fillstyle=solid,fillcolor=verylight}
\definecolor{verylight}{gray}{0.97}
\definecolor{light}{gray}{0.9}
\definecolor{medium}{gray}{0.85}
\definecolor{dark}{gray}{0.6}

 \def\NZQ{\mathbb}               % the font for N,Z,Q,R,C
 \def\NN{{\NZQ N}}
 
 \def\ZZ{{\NZQ Z}}
 \def\RR{{\NZQ R}}

 \def\G{{\mathcal G}}

 \def\P{{\mathcal P}}
  
  \def\Hc{{\mathcal H}}

 \def\ab{{\mathbf a}}
 
 \def\xb{{\mathbf x}}

 \def\eb{{\mathbf e}}
 \def\opn#1#2{\def#1{\operatorname{#2}}} % to make operators
 \opn\chara{char} \opn\length{\ell} \opn\pd{pd} \opn\rk{rk}
 \opn\projdim{proj\,dim} \opn\injdim{inj\,dim} \opn\rank{rank}
 \opn\depth{depth} \opn\grade{grade} \opn\height{height}
 \opn\embdim{emb\,dim} \opn\codim{codim}
 
 \opn\Tr{Tr} \opn\bigrank{big\,rank}
 \opn\superheight{superheight}\opn\lcm{lcm}
 \opn\trdeg{tr\,deg}%\emph{
 \opn\reg{reg} \opn\lreg{lreg} \opn\ini{in} \opn\lpd{lpd}
 \opn\size{size} \opn\sdepth{sdepth}
 \opn\link{link}\opn\fdepth{fdepth}\opn\lex{lex}
 \opn\tr{tr}
 \opn\type{type}
 \opn\gap{gap}
 \opn\arithdeg{arith-deg}
 \opn\astab{astab}
  \opn\dstab{dstab}
  \opn\pol{pol}
  \opn\mat{mat}
  \opn\indmat{indmat}
 \opn\div{div} \opn\Div{Div} \opn\cl{cl} \opn\Cl{Cl}
 \opn\Spec{Spec} \opn\Supp{Supp} \opn\supp{supp} \opn\Sing{Sing}
 \opn\Ass{Ass} \opn\Min{Min}\opn\Mon{Mon}
 \opn\Ann{Ann} \opn\Rad{Rad} \opn\Soc{Soc}
 \opn\Im{Im} \opn\Ker{Ker} \opn\Coker{Coker} \opn\Am{Am}
 \opn\Hom{Hom} \opn\Tor{Tor} \opn\Ext{Ext} \opn\End{End}
 \opn\Aut{Aut} \opn\id{id}
 
 \opn\nat{nat}
 \opn\pff{pf}%   \pf exists already
 \opn\Pf{Pf} \opn\GL{GL} \opn\SL{SL} \opn\mod{mod} \opn\ord{ord}
 \opn\Gin{Gin} \opn\Hilb{Hilb}\opn\sort{sort}
 \opn\PF{PF}\opn\Ap{Ap}
 \opn\mult{mult}
 \opn\aff{aff}
 \opn\relint{relint} \opn\st{st}
 \opn\lk{lk} \opn\cn{cn} \opn\core{core} \opn\vol{vol}  \opn\inp{inp} \opn\nilpot{nilpot}
 \opn\link{link} \opn\star{star}\opn\lex{lex}\opn\set{set}
 \opn\width{wd}
 \opn\Fr{F}
 \opn\QF{QF}
 \opn\G{G}
 \opn\type{type}\opn\res{res}
 \opn\conv{conv}
 \opn\gr{gr}
 
 \def\pot#1#2{#1[\kern-0.28ex[#2]\kern-0.28ex]}

 \opn\dirlim{\underrightarrow{\lim}}
 \opn\inivlim{\underleftarrow{\lim}}
 \let\union=\cup
 \let\sect=\cap

 \let\iso=\cong

 \let\Dirsum=\bigoplus
 
 \def\Implies{\ifmmode\Longrightarrow \else
         \unskip${}\Longrightarrow{}$\ignorespaces\fi}
 \def\implies{\ifmmode\Rightarrow \else
         \unskip${}\Rightarrow{}$\ignorespaces\fi}
 \def\iff{\ifmmode\Longleftrightarrow \else
         \unskip${}\Longleftrightarrow{}$\ignorespaces\fi}

 \let\:=\colon
 \newtheorem{Theorem}{Theorem}[section]
 \newtheorem{Lemma}[Theorem]{Lemma}
 \newtheorem{Corollary}[Theorem]{Corollary}
 \newtheorem{Proposition}[Theorem]{Proposition}
 
 \newtheorem{Remarks}[Theorem]{Remarks}
 \newtheorem{Example}[Theorem]{Example}

 \let\epsilon\varepsilon
 \let\kappa=\varkappa
 \def\qed{\ifhmode\textqed\fi
       \ifmmode\ifinner\quad\qedsymbol\else\dispqed\fi\fi}
 \def\textqed{\unskip\nobreak\penalty50
        \hskip2em\hbox{}\nobreak\hfil\qedsymbol
        \parfillskip=0pt \finalhyphendemerits=0}
 \def\dispqed{\rlap{\qquad\qedsymbol}}
 
 \opn\dis{dis}
 \def\pnt{{\raise0.5mm\hbox{\large\bf.}}}
 
 \opn\Lex{Lex}

\begin{document}

\title {Matching numbers and the regularity of the Rees algebra of an edge ideal}

\author {J\"urgen Herzog and Takayuki Hibi}

\address{J\"urgen Herzog, Fachbereich Mathematik, Universit\"at Duisburg-Essen, Campus Essen, 45117
Essen, Germany} \email{juergen.herzog@uni-essen.de}

\address{Takayuki Hibi, Department of Pure and Applied Mathematics, Graduate School of Information Science and Technology,
Osaka University, Suita, Osaka 565-0871, Japan}
\email{hibi@math.sci.osaka-u.ac.jp}

\dedicatory{ }

\begin{abstract}
The regularity of the Rees ring of the edge ideal of a finite simple graph is studied. We show that the matching number is a lower and matching number~$+1$ is an upper bound of the regularity, if the Rees algebra is normal. In general the induced matching number is a lower bound for the regularity,   which  can be shown by applying the squarefree divisor complex.
\end{abstract}

\thanks{ The second author was partially supported by JSPS KAKENHI 19H00637}

%		13H10   	Special types (Cohen-Macaulay, Gorenstein, Buchsbaum, etc.)
%		13D02   	Syzygies, resolutions, complexes
%		05E40   	Combinatorial aspects of commutative algebra
%		16S36   	Ordinary and skew polynomial rings and semigroup rings

%		14M25   	Toric varieties, Newton polyhedra [See also 52B20]
%		13A02   	Graded rings
%		13F20   	Polynomial rings and ideals; rings of integer-valued polynomials
%		13A18   	Valuations and their generalizations
%		06A11   	Algebraic aspects of posets


\maketitle

\setcounter{tocdepth}{1}
%\tableofcontents
\section*{Introduction}

In the study of powers of monomial ideals,   the Rees algebra  plays an important role. In the present paper we focus on the  Rees algebra of the edge ideal of a finite simple graph.

Let $K$ be a field and $S=K[x_1,\dots,x_n]$ be the polynomial ring over $K$ in the  variables $x_1,\ldots,x_n$. Furthermore, let $G$  be a simple graph on the vertex set $V(G)=[n]$ and edge set $E(G)$. The edge ideal $I=I(G)$ of $G$ is the monomial ideal in $S$ generated by the monomials $x_ix_j$ with $\{i,j\}\in E(G)$, and the edge ring $K[G]$ is the $K$-algebra generated by the monomials $x_ix_j\in I$. The Rees algebra $R(I)=\Dirsum_{s\geq 0}I^s$ of $I$ may be viewed as the edge ring of the graph $G^*$ with  $V(G^*)=[n+1]$  and $E(G^*)=E(G)\union\{\{i,n+1\}\:\; i\in V(G)\}$.  We are interested in bounding the regularity of $R(I)$, because this provides information about the regularity  of the powers of $I$.
In our situation, computing regularity of $R(I)$ amounts to compute the regularity of the edge ring of graphs of the form $G^*$. This class of algebras are particular classes of toric rings.  To our knowledge,  there are essentially  two methods available to compute the regularity of a toric ring $A$.  The first method, which can always be applied, is to compute the multigraded Betti numbers of $A$ by using the {\em squarefree divisor complex} (see \cite[Proposition 1.1]{BH1} and \cite[Theorem 3.28]{HHO}). In concrete cases, as discussed in Section 1, this allows to give lower bounds for  the regularity, but in general it is hard to use. The second method can be applied, if $A$ is Cohen--Macaulay, in which case one needs to compute the $a$-invariant of $A$. If in addition,  $A$ is normal, then following  Danilov and Stanley \cite[Theorem 6.3.5]{BH},  the canonical  module can computed which in particular gives us the $a$-invariant of $A$.

Let $A$ be a toric ring generated by monomials in $S$. A $K$-subalgebra $B$ of $A$ is called a {\em combinatorial pure} subalgebra of $A$ if there exists a subset $T\subset [n]$ such that $B=A\sect K[x_i\:\; i\in T]$.  In Section 1 we recall  squarefree divisor complexes and use them to prove that if $B\subset A$ is a combinatorial pure subring of $A$, then  $\beta_{i,j}(B)\leq  \beta_{i,j}(A)$ for all $i$ and $j$. Here the $\beta_{i,j}(M)$ denote  the graded Betti numbers of a graded module $M$. This result implies in particular that if $I$ is a monomial ideal generated in a single degree, then  $\reg R(I)\geq \reg F(I)$, where $F(I)$ is the fiber cone of $I$.  We also use squarefree divisor complexes to give lower bounds of the regularity of the Rees algebra of $I(G)$, when $G$ is a disjoint union of edges. These  results will be used in the next section.

Now, we devote Section $2$ to finding an upper bound and a lower bound of the regularity of the Rees algebra $R(I(G))$ of the edge ideal $I(G)$ of a finite simple graph $G$ in terms of the matching number of $G$ and the induced matching number of $G$.  Recall that a {\em matching} of $G$ is a subset $M \subset E(G)$ such that $e \cap e' = \emptyset$ for all $e$ and $e'$ belonging to $M$ with $e \neq e'$.  A matching $M$ of $G$ is called {\em perfect} if, for each $i \in [n]$, there is $e \in M$ with $i \in e$.  An {\em induced matching} of $G$ is a matching $M$ of $G$ such that if $e$ and $e'$ belong to $M$ with $e \neq e'$, then there is no edge $f \in E(G)$ with $e \cap f \neq \emptyset$ and $e' \cap f \neq \emptyset$.  The maximal  of matchings of $G$ is called the {\em matching number} of $G$ and is denoted by $\mat(G)$.  The {\em induced matching number} of $G$, denoted by $\indmat(G)$, is the maximal cardinality of induced matchings of $G$.

It is known \cite[Corollary 2.3]{OH} that $K[G]$ is normal if and only if each connected component $G'$ of $G$ satisfies the {\em odd cycle condition}, which says that if $C$ and $C'$ are odd cycles of $G'$ with $V(C) \cap V(C') = \emptyset$, then there are $i \in V(C)$ and $j \in V(C')$ with $\{i, j\} \in E(G')$.  In particular, if $G$ is bipartite, then $K[G]$ is normal.  It is shown \cite[Theorem 4.2]{Cid} that, when $G$ is bipartite, one has $\reg R(I(G)) = \mat(G)$.

Our main result (Theorem \ref{main}) says that, if $R(I(G))$ is normal, then
\[
\mat(G) \leq \reg R(I(G)) \leq \mat(G) + 1.
\]
Furthermore, if $G$ has a perfect matching, then $\reg R(I(G)) =\mat(G)$.  Our proof heavily depends on the theory of normal edge polytopes created in \cite{OH} as well as on the result \cite[Theorem 3.3]{Sta} which says that if $G'$ is a subgraph of $G$ and if each of $K[G']$ and $K[G]$ is normal, then one has $\reg K[G'] \leq \reg K[G]$.  Even though $K[G]$ is normal, the above upper bound may not be satisfied (Remarks~\ref{maybe}).

On the other hand, it follows from Proposition~\ref{pure} that, for any finite simple graph $G$, one has $\indmat(G) \leq \reg R(I(G))$.  We, however, very much believe that the inequality $\mat(G) \leq \reg R(I(G))$ is valid for any finite simple graph $G$.

\section{Combinatorial pure subrings and regularity}

Let $K$ be a field and $A=K[u_1,\ldots,u_m]\subset S=K[x_1,\ldots,x_n]$ be the $K$-algebra minimally generated by the monomials $u_i=\xb^{\ab_i}$ in the polynomial ring $S$. Here for $\ab=(a_1,\ldots, a_n)$ we denote by $\xb^\ab$ the monomial $x_1^{a_1}x_2^{a_2}\cdots x_n^{a_n}$.

The $K$-algebra $A$   has a $K$-basis consisting of monomials $\xb^{\ab}$.  The set of exponents $\ab$ appearing  as exponents of the basis elements of $A$ together with addition form a positive affine semigroup  $H\subset \NN^n$ which is generated by $\ab_1,\ldots,\ab_n$.

Given an element $\ab \in H$,  we define the simplicial complex
\[
\Delta_\ab(A)=\{F\subset [n]\:\; u^F \text{ divides } \xb^{\ab} \text{ in } A\}.
\]
where $u^F=\prod_{j\in F}u_j$.

The simplicial complex $\Delta_{\ab}(A)$ is called the \emph{squarefree divisor complex} of $H$ (or of $A$) with respect to $\ab$.

We recall the following  theorem from \cite[Proposition 1.1]{BH1} (see also \cite[Theorem~3.28]{HHO})

\begin{Theorem}
\label{squarefreedivisorcomplex}
For the multigraded Betti numbers of $A$ one has
\[
\beta_{i,\ab}(A)=\dim_K  \widetilde{H}_{i-1}(\Delta_\ab; K).
\]
Here $\widetilde{H}_{i}(\Gamma; K)$ denotes the $i$th reduced simplicial homology of a simplicial complex~$\Gamma$.
\end{Theorem}

We demonstrate this theorem by a simple example. Let $A=K[x_1^2,x_1x_2,x_2^2]\subset K[x_1,x_2]$.  Then $u_1=x_1^2$, $u_2=x_1x_2$ and $u_3=x_2^2$ and $H\subset \ZZ^2$ is generated by $(2,0)$, $(1,1)$ and  $(0,2)$. We want to compute $\beta_{1,(2,2)}(A)$. Squarefree divisors of $x_1^2x_2^2$ are $u_1u_3$ and $u_2$. Therefore,  the facets of $\Delta_{(2,2)}$ are $\{1,3\}$ and $\{2\}$. Thus, $\Delta_{(2,2)}$ has two connected components and so
$\dim_K  \widetilde{H}_{0}(\Delta_{(2,2)}; K)=1$. For the total standard grading this means that $\beta_{1,2}=1$. Indeed the Betti diagram of $A$ is the following:

\begin{verbatim}
        0    1
---------------
 0:     1    -
 1:     -    1
---------------
Tot:    1    1
\end{verbatim}

\begin{Proposition}
\label{pure}
Let $B\subset A$ a combinatorial pure subring of the toric $K$-algebra $A$. Then $\beta_{i,\ab}(B)= \beta_{i,\ab}(A)$ for all $\ab$ with  $\xb^\ab \in B$.
\end{Proposition}

\begin{proof}
Let $A=K[u_1,\ldots,u_m]$ be minimally generated by the monomials $u_1,\ldots,u_m$ in $K[x_1,\ldots,x_n]$, and assume that
$B=A\sect K[x_i\:\; i\in T]$ for some subset $T\subset [n]$. We may further assume that $\supp(u_i)\subset T$ for $i\leq r$ and $\supp(u_i)\not \subset T$ for $i>r$. Then $B=K[u_1,\ldots,u_r]$. We claim that $\Delta_\ab(B)=\Delta_\ab(A)$. Indeed, let $F\in \Delta_\ab(B)$.  Then $\prod_{i\in F}u_i$ divides $\xb^\ab$ in $B$. Then it also divides $\xb^\ab$  in $A$. This shows that $\Delta_\ab(B)\subset \Delta_\ab(A)$. Conversely, let $F\in \Delta_\ab(A)$.  Then $\prod_{i\in F}u_i$ divides $\xb^\ab$ in $A$. Since $\supp(\xb^\ab)\in T$, it follows that all $u_i$ with $i\in F$ belong to $B$, and since $\prod_{i\in F}u_i$ divides $\xb^\ab$ in $A$, we have $\xb^{\ab}=(\prod_{i\in F}u_i) \prod_{j=1,\ldots,m}u_j^{a_j}$ for some integers $a_j\geq 0$. Since $\supp(\xb^\ab)\subset T$, the same holds true for all $u_j$ with $a_j>0$. It follows that $\prod_{j=1,\ldots,m}u_j^{a_j}\in B$, and hence $F\in \Delta_\ab(B)$.
\end{proof}

\begin{Corollary}
\label{equigenerated}
Let $A$ be a toric ring with all generators of same degree $d$, and let $B$ be a combinatorial pure subring. Then $A$ and $B$ are naturally standard graded, and we have $\beta_{i,j}(B)\leq \beta_{i,j}(A)$ for all $i$ and $j$.
\end{Corollary}

\begin{proof} For $\ab=(a_1,\ldots,a_n)\in\ZZ^n$ we set $|\ab|=\sum_{i=1,\ldots}a_i$. Let $H=\{\ab\:\; \xb^{\ab}\in B\}$ and $H'=\{\ab\:\; \xb^{\ab} \in A\}$. Then
\[
\beta_{i,j}(B) =\sum_{\ab\in H,\; |\ab|=dj}\beta_{i,\ab}(B)\leq \sum_{\ab\in H',\; |\ab|=dj}\beta_{i,\ab}(A)=\beta_{i,j}(A).
\]
\end{proof}

\begin{Corollary}
\label{fiber}
Let $G$ be a finite simple graph and $I=I(G)$ its edge ideal. Then $\beta_{i,j}(R(I))\geq \beta_{i,j}(K[G])$ for all $i$ and $j$. In particular, $\reg(R(I))\geq \reg (K[G])$.
\end{Corollary}

\begin{proof}
The Rees ring $R(I)$ is isomorphic to $K[G^*]$. Since $K[G]$ is a combinatorial pure subring of $K[G^*]$,  the assertion follows from Corollary~\ref{equigenerated}.
\end{proof}

In the next two proposition we consider the Rees algebra of a  special graph which plays a role in the next section.

\begin{Proposition}
\label{edges}
Let $G$ be the graph consisting of $m$ disjoint edges, and let $I=I(G)$ be the edge ideal of $G$. Then $\reg R(I)=0$  if $m=1$  and $\reg R(I)\geq m$ if $m>1$.
\end{Proposition}

\begin{proof}
If $m=1$, then $E(G)=\{1,2\}$ and $E(G^*)=\{\{1,2\},\{1,3\},\{2,3\}\}$, and hence $G^*$ is a $3$-cycle. Hence $R(I)$ which is isomorphic to $K[G^*]$ is a polynomial ring, since the genertors of $K[G^*]$ are algebraically independent.

Now let $m>1$, and for $i=1,\ldots,m$ let $e_i=\{2i-1,2i\}$  be the edges of $G$. Then $G^*$ has the additional edges $f_i=\{i,2m+1\}$  $(i=1,\ldots, 2m)$. Let the  $u_i =x_{2i-1}x_{2i}$ be the monomial generators of $K[G^*]$ corresponding to the edges $e_i$ and  the $v_i=x_ix_{2m+1}$ the generators  of $K[G^*]$ corresponding to the edges $f_i$. In $K[G^*]$ we consider the element $\xb^{\ab}=u_1\prod_{i=3,\ldots, 2m}v_i=x_1x_2\ldots x_{2m}x_{m+1}^{2m-2}$, and claim that
\[
\dim_K\widetilde{H}_{m-2}(\Delta_\ab)\neq 0.
\]
Thus, by Theorem~\ref{squarefreedivisorcomplex} the claim implies that $\beta_{{m-1},\ab}(R(I))\ne 0$.  Since $|\ab|=2(2m-1)$,  it follows that
$\beta_{{m-1},2m-1}(R(I))\geq 1$, and this implies that $\reg(R(I))\geq m$.

\medskip
Proof of the claim: we first notice that $\Delta_\ab$  has the facets
\[
F_i=\{u_i, v_1,\ldots, v_{2m}\}\setminus \{v_{2i-1},v_{2i}\} \quad \text{for} \quad i=1,\ldots, m.
\]
where   the vertices of $\Delta_\ab$ are identified with the monomials $u_i$ and $v_j$.

Indeed, since $u_i\prod_{j=1,\ldots,2m \atop j\neq 2i-1, 2i}v_j=\xb^\ab$ for $i=1,\ldots, m$ it follows that $F_1, \ldots, F_m$ are facets of $\Delta_\ab$.

In order to prove that these are all the facets of $\Delta_\ab$, we show that if $F$ is a face of $\Delta_\ab$, then $F\subset F_i$ for some $i$. Suppose $\{e_i,e_j\}\subset F$ for some $i\neq j$. By symmetry we may assume that $\{e_1,e_2\}\subset F$. Then $\xb^{\ab}/x_1x_2x_3x_4=x_5\cdots x_{2m}x_{m+1}^{2m-2}\in A$ which is impossible because  $x_5\cdots x_{2m}x_{m+1}^{2m-2}$ must contain $2m-2$ factor of the $f_j$. Next suppose that $F\sect \{e_1,\ldots,e_m\}=\emptyset$. If there exists $i$ such that $F\sect \{f_{2i-1},f_{2i}\}=\emptyset$, then $F\subset F_i$. Otherwise, $F\sect \{f_{2i-1}, f_{2i}\}\neq \emptyset$ for all $i$, and by symmetry we may assume that  $f_{2i-1}\in F$ for $i=1,\ldots, m$. Since $u^F$ divides $\xb^{\ab}$ it follows that $|F|\leq 2m-1$. By symmetry we may assume that $f_2$ and $f_4$  do not belong to $F$. Then $\xb^{\ab}/u^F= x_2x_4\prod_{i\;  f_{2i}\not\in F}x_{2i}\not \in A$. Hence, $F\not \in \Delta_{\ab}$. It remains to consider the case that $F\sect\{e_1,\ldots, e_m\}=\{e_i\}$  for some $i$. By symmetry we may assume that $i=1$. Suppose that $f_1\in F$.  Then $u_1f_1=x_1^2x_2x_{m+1}$ divides $\xb^{\ab}$, a contradiction. Thus $f_1\not\in F$. Similarly, $f_2\not \in F$. This shows that $F\subset F_1$.

Next we notice  that geometric realization $|\Delta_\ab|$ of $\Delta_\ab$ is homotopic to the geometric realization of the simplicial complex $\Gamma$ whose facets are
\[
G_i=\{f_1,f_3,\ldots, f_{2m-1}\}\setminus \{f_{2i-1}\}, \quad i=1,\ldots,m.
\]
We choose the standard geometric realization by  identifying the vertices \[e_1,\ldots,e_m, f_1,\ldots,f_{2m}\] of $\Delta_\ab$ with the standard unit vectors in $\RR^{3m}$.

Indeed, the homotopy is given   by the affine maps $\varphi_t$ induced  by  $e_i\mapsto e_i'=te_i+(1-t)f_1$ if $i>1$ and  $e_1\mapsto e_1'=te_1+(t-1)f_3$. Moreover, $f_i\mapsto f_i'=f_i$ if $i$ is odd  and $f_i\mapsto f_i'=tf_{i}+(1-t)f_{i-1}$ if $i$ is even. Here $0\leq t\leq 1$.
We have  $\varphi_1=\id$ and $|\Gamma|= \varphi_0(|\Delta_\ab|)$, as desired.

Now since $|\Gamma|$ is homotopic to $|\Delta_\ab|$, we see that  $\widetilde{H}_{m-2}(\Gamma)\iso \widetilde{H}_{m-2}(\Delta_\ab)$. Observe  that  $|\Gamma|$ is homotopic to an $(m-2)$-sphere, so that $\widetilde{H}_{m-2}(\Gamma)\neq 0$. This concludes the proof of the proposition.
\end{proof}

\begin{Remarks}
\label{maybe}
{\em Let $G$ be the sum of the graphs $G_1$ and $G_2$. Assume that $G$ has no isolated vertices and $G_1$ or $G_2$ has at least 2 edges. Considering several examples we come up with the following question: Is it true that $\reg R(I(G_1+G_2))= \reg R(I(G_1))+\reg R(I(G_2))$?

 If this question has a positive answer, then Proposition~\ref{edges} is just a very special case of this statement. Of course it is also a simple consequence of the theorem of Cid-Ruiz \cite{Cid}. However, in order to keep this paper as self-contained as possible and also to demonstrate the use of  squarefree divisor complexes,  we included  Proposition~\ref{edges} to this paper.

With CoCoA we considered  the case that $G$ is the sum of two  $3$-cycles, and found $\reg R(I(G))=4$, as expected. In the next section we show  that $\reg R(I(G))\leq \mat(G)+1$ if $R(IG))$ is normal. In our  example with the two $3$-cycles, $R(I(G))$ is not normal and $\mat(G)=2$. Therefore,  the inequality $\reg R(I(G))\leq \mat(G)+1$ is in general not valid if $\reg R(I(G))$ is not normal.

If our question has a positive answer, then one has $\reg R(I(G))=2m$ if $G$ is the sum of $m$ $3$-cycles. On the other hand, for this graph,  $\mat(G)=m$. This then gives a family of graphs for which $\reg R(I(G))-\mat(G)$ can be any positive integer. }
\end{Remarks}

\section{Bounds for the regularity of the Rees algebra of an edge ideal}

Let, as before, $G$ be a finite simple graph on the vertex $V(G) = [n]$ and $S = K[x_1, \ldots, x_n]$ the polynomial ring in $n$ variables over a field $K$.  Let $\P_G \subset \RR^n$ be the {\em edge polytope} of $G$ which is the convex hull of $\{ \, \eb_i + \eb_j \, : \, \{i, j\} \in E(G) \, \}$, where $\eb_i$ is the $i$th unit coordinate vector of $\RR^n$.  Let $A_{\P_G}$ denote the {\em Ehrhart ring} of $G$, which is the toric ring in the $n + 1$ variables $x_1, \ldots, x_n, t$ whose $K$-basis consists of those monomials $x_1^{a_1} \ldots x_1^{a_1} t^q$, where $1 \leq q \in \ZZ$, with $(a_1, \ldots, a_n) \in q \P_G \cap \ZZ_{\geq 0}^n$.  We refer the reader to \cite{OH} for basic materials and fundamental results on edge polytopes and their Ehrhart rings.  The Ehrhart ring $A_{\P_G}$ is normal and its canonical module is spanned by those monomials $x_1^{a_1} \ldots x_1^{a_1} t^q$ with $(a_1, \ldots, a_n) \in q (\P_G \setminus \partial \P_G) \cap \ZZ_{\geq 0}^n$.  The edge ring $K[G]$ is normal if and only if $A_{\P_G}$ is standard grading, i.e., $A_{\P_G}$ is generated by those monomials $x_1^{a_1} \ldots x_1^{a_1} t$ with $(a_1, \ldots, a_n) \in \P_G \cap \ZZ_{\geq 0}^n$ as an algebra over $K$.  Thus in particular $K[G]$ is normal if and only if $K[G]$ is isomorphic to $A_{\P_G}$.  Furthermore, it is shown \cite[Corollary 2.3]{OH} that $K[G]$ is normal if and only if each connected component $G'$ of $G$ satisfies the {\em odd cycle condition}, that is, if $C$ and $C'$ are odd cycles of $G'$ with $V(C) \cap V(C') = \emptyset$, then there exist $i \in V(C)$ and $j \in V(C')$ with $\{i, j\} \in E(G')$.  In particular, if $G$ is bipartite, then $K[G]$ is normal.

We say that a finite subset $L \subset E(G)$ is an {\em edge cover} of $G$ if $\cup_{e \in L} = [n]$.  Let $\mu(G)$ denote the minimal cardinality of edge covers of $G$.

\begin{Lemma}
\label{edge}
Let $G$ be a finite simple graph on $V(G) = [n]$.  Then
\[
\mu(G) + \mat(G) = n.
\]
\end{Lemma}

\begin{proof}
Let $L \subset E(G)$ be an edge cover with $|L| = \mu(G)$ and $M'$ a matching of $G$ which is maximal among those matchings $M$ with $M \subset L$.  Then, for each edge $e \in L \setminus M$, there is $f \in M$ with $e \cap f \neq \emptyset$. Hence $2|M| + (\mu(G) - |M|) = n$.  Thus $\mat(G) \geq |M| = n - \mu(G)$.  Hence $\mu(G) \geq n - \mat(G)$.  However, clearly, one has $\mu(G) \leq n - \mat(G)$.  Thus $\mu(G) = n - \mat(G)$, as desired.
\end{proof}

In order to achieve the proof of Theorem \ref{main}, the information of the facets of edge polytopes is indispensable.  Let $G$ be a connected non-bipartite graph on $V(G) = [n]$.  We say that $i \in[n]$ is {\em regular} (\cite[p.~414]{OH}) if each connected component of the induced subgraph $G_{[n] \setminus \{ i \}}$ is non-bipartite.  When $i \in [n]$ is regular, the hyperplane $\Hc_i$ of $\RR^n$ defined by the equation $z_i = 0$ is a supporting hyperplane of $\P_G$ and $\Hc_i \cap \P_G$ is a facet of $\P_G$ (\cite[Theorem 1.7]{OH}).  A nonempty subset $T \subset [n]$ is called {\em independent} if $\{i, j\} \not\in E(G)$ for $i$ and $j$ belonging to $T$ with $i \neq j$.  When $T$ is independent, we write $N_G(T) \subset [n]$ for the set of those vertices $i \in [n]$ for which there is an edge $e \in E(G)$ with $i \in e$ and $e \cap T \neq \emptyset$.  When $T$ is independent, we write $T^\sharp$ for the bipartite graph on the vertex set $T \cup N_G(T)$ whose edges are those $\{i, j\} \in E(G)$ with $i \in T$ and $j \in N_G(T)$.
When $T$ is independent, we say that $T$ is {\em fundamental} (\cite[p.~415]{OH}) if (i) $T^\sharp$ is connected and (ii) either $T \cup N_G(T) = [n]$ or each connected component of the induced subgraph $G_{[n] \setminus (T \cup N_G(T))}$ is non-bipartite.  When $T$ is fundamental, the hyperplane $\Hc_T$ of $\RR^n$ defined by the equation $\sum_{i \in T} z_i = \sum_{j \in N_G(T)} z_j$ is a supporting hyperplane of $\P_G$ and $\Hc_T \cap \P_G$ is a facet of $\P_G$ (\cite[Theorem 1.7]{OH}).

We now come to the main result of the present paper.  It is known \cite[Theorem 4.2]{Cid} that, when $G$ is bipartite, one has $\reg R(I(G)) = \mat(G)$. % We also give a combinatorial proof of the fact.

\begin{Theorem}
\label{main}
{\rm (a)}  Let $G$ be a finite simple graph with $G_1, \ldots, G_c$ its connected components.  Then the Rees algebra $R(I(G)) = K[G^*]$ is normal if and only if each $K[G_i]$ is normal and at most one of $G_1, \ldots, G_c$ is non-bipartite.

{\rm (b)}  Let $|E(G)| \geq 2$. Suppose that $R(I(G))$ is normal.  Then
\[
\mat(G) \leq \reg R(I(G)) \leq \mat(G) + 1.
\]
% Furthermore, if $G$ is bipartite, then
% \[
% \reg R(I(G)) = \mat(G).
% \]
\end{Theorem}

\begin{proof}
(a) The  ``If'' part is clear.  We show  now the ``Only If'' part.  If, say, $G_1$ fails to satisfy the odd cycle condition, then $G^*$ also fails to satisfy the odd cycle condition.  If, say, $G_1$ and $G_2$ are non-bipartite and if $C_i$ is an odd cycle of $G_i$ for $i \in \{1, 2\}$, then, even though $G^*$ is connected, there is no edge $e \in E(G^*)$ with $e \cap V(C_i) \neq \emptyset$ for $i \in \{1, 2\}$, as desired.

\medskip

(b) We first prove the lower bound. In case   $\mat(G)=1$, the graph $G$ is a star graph which by assumption has at least 2 edges.  Then $R(I(G))$ is not polynomial and therefore $\reg R(I(G))\geq  1= \mat(G)$.

Now let $m = \mat(G)$ and assume that $m \geq 2$ and $\{\{i_1,j_1\}, \ldots, \{i_m,j_m\}\}$ a matching of $G$.  Write $H$ for the subgraph of $G$ with $E(H) = \{\{i_1,j_1\}, \ldots, \{i_m,j_m\}\}$.  Since $R(I(H)) = K[H^*]$ is normal with $\reg R(I(H)) = m$ (Proposition \ref{edges}) and since $H^*$ is a subgraph of $G^*$, it follows from \cite[Theorem 3.3]{Sta} that
\[
\mat(G) = \reg R(I(H)) \leq \reg R(I(G)).
\]

We now prove the upper bound. The highlight of the proof is to estimate the positive integer
\[
q_0 = \min \{ \, q \geq 1 \, : \, q (\P_{G^*} \setminus \partial \P_{G^*}) \cap \ZZ^{n+1} \neq \emptyset \, \}.
\]

\medskip

(Case I)
Suppose that $G$ is connected and non-bipartite.  Each vertex $i \in [n + 1]$ of $\P_{G^*}$ is regular.  It then follows from \cite[Theorem 1.7]{OH} that, if $(a_1, \ldots, a_n, a_{n+1})$ belongs to $q (\P_{G^*} \setminus \partial \P_{G^*}) \cap \ZZ^{n+1}$, then each $a_i > 0$.  By using Lemma \ref{edge} one has
\[
q_0 \geq \mu(G) = (n + 1) - \mat(G^*) \geq (n + 1) - (\mat(G) + 1).
\]
Since $\dim \P_{G^*}  = n$, it follows that
\[
\reg R(I(G)) = (n + 1) - q_0 \leq \mat(G) + 1,
\]
as desired.

\medskip

% Suppose that $G$ is bipartite with the decomposition $V(G) = V_1 \cup V_2$, where $V_1 = \{1, \ldots, r\}$ and $V_2 = \{r + 1, \ldots, n\}$ with $r \leq n - r$.  Then $\mat(G) \leq r$.  Each $i \in [n]$ is regular.  However, since $G$ is bipartite, the vertex $n + 1$ cannot be regular.  Again, by \cite[Theorem 1.7]{OH}, if $(a_1, \ldots, a_n, a_{n+1}) \in q (\P_{G^*} \setminus \partial \P_{G^*}) \cap \ZZ^{n+1}$, then $a_i > 0$ for $1 \leq i \leq n$.  Thus Lemma \ref{edge} says that $q_0 \geq n - r$.  Let $\ab = (a_1, \ldots, a_n, a_{n+1})$ belong to $(n - r) (\P_{G^*} \setminus \partial \P_{G^*}) \cap \ZZ^{n+1}$.  It then follows that $a_{n+1} = 0$ and $(\sum_{i=1}^{r} a_i) + a_n = \sum_{j=r+1}^{n+1} a_j$.  Since the subset $V_2$ of $V(G^*) = [n+1]$ is fundamental in $G_{r, n-r}^*$, it follows from \cite[Theorem 1.7]{OH} that $\ab$ cannot belong to $(n - r) (\P_{G^*} \setminus \partial \P_{G^*}) \cap \ZZ^{n+1}$.  In other words, $(n - r) (\P_{G^*} \setminus \partial \P_{G^*}) \cap \ZZ^{n+1} = \emptyset$.  Hence $q_0 \geq n - r + 1$.  Hence
% \[
% \reg R(I(G_{r, n-r}^*)) = (n + 1) - q_0 \leq r \leq \mat(G),
% \]
% as desired.

% \medskip

(Case II)
Suppose that $G$ is disconnected and non-bipartite.  Let $G_1, \ldots G_c$ be the connected components of $G$, where $G_1$ is non-bipartite and where each of $G_2, \ldots, G_c$ is bipartite.  Let $V(G_i) = V_i \cup V_i'$ be the decomposition of $V(G_i)$.  Let $(a_1, \ldots, a_n, a_{n+1})$ belong to $q (\P_{G^*} \setminus \partial \P_{G^*}) \cap \ZZ^{n+1}$.  Since each $i \in [n]$ is regular, one has $a_i > 0$ for each $i \in [n]$.  Since $T = V_2 \cup \cdots \cup V_c$ is fundamental with $N_G(T) = V'_2 \cup \cdots \cup V'_c \cup \{ n + 1 \}$, it follows that $a_{n+1} > 0$.  Hence, as in (Case I), one has $\reg R(I(G)) \leq \mat(G) + 1$, as required.
\end{proof}

In the proof of (b) of Theorem \ref{main}, one has $\mat(G^*) = \mat(G)$ if and only if $G$ has a perfect matching.  As a result,

\begin{Corollary}
\label{perfect}
If $G$ has a perfect matching and if $R(I(G))$ is normal, then
\[
\reg R(I(G)) = \mat(G).
\]
\end{Corollary}

The converse of Corollary \ref{perfect} is false.  In fact,

\begin{Example}
{\em
Let $G$ be a finite simple connected non-bipartite graph on $[6]$ whose edges are
\[
\{1,2\}, \{2,3\}, \{3,4\}, \{1,4\}, \{2,4\}, \{2,5\}, \{2,6\}.
\]
Even though $G$ has no perfect matching, one has $\reg R(I(G)) = \mat(G) = 2$.  The lattice points $(a_1, \ldots, a_7)$ belonging to $4 \P_{G^*} \cap \ZZ^{7}$ with each $a_i > 0$ are
\[
(1,1,1,1,1,1,2), \, (1,1,1,2,1,1,1).
\]
Since $T = \{1,3,5,6\}$ is fundamental, neither of these lattice points cannot belong to $4(\P_{G^*} \setminus \partial \P_{G^*})$.  Thus $q_0 \geq 5$.
}
\end{Example}

It would, of course, be of interest to characterize finite simple graphs $G$ with $\reg R(I(G)) = \mat(G)$. On the other hand, if $G$ is a $5$-cycle, then $R(I(G))$ is normal, $\reg R(I(G))=3$ and $\mat(G)=2$.

\medskip
When $K[G]$ is non-normal, instead of \cite[Theorem 3.3]{Sta}, we can enjoy the merit of combinatorial pure subrings (Corollary \ref{equigenerated}).

\begin{Proposition}
\label{cps}
Let $G$ be an arbitrary finite simple graph.  Then one has
\[
\indmat(G) \leq \reg R(I(G)).
\]
\end{Proposition}

We, however, believe that the lower bound inequality $\mat(G) \leq \reg R(I(G))$ is valid for any finite simple graph $G$.

\end{document}